\documentclass[11pt]{amsart}
\usepackage[all,knot]{xy}

\usepackage[active]{srcltx}

\usepackage{pb-diagram}

\usepackage[usenames,dvipsnames,svgnames,table]{xcolor}

\usepackage{mathrsfs}
\usepackage{amsmath}
\usepackage{amssymb}
\usepackage{amscd}
\usepackage{amsthm}
\usepackage[latin1]{inputenc}
\usepackage{graphics}
\usepackage{varioref}

\labelformat{enumi}{(#1)}

\newtheorem{thm}{Theorem} [section]
\newtheorem{defin}{Definition} [section]
\newtheorem{remark}{Remark}[section]
\newtheorem{prop}{Proposition}[section]
\newtheorem{cor}{Corollary}[section]
\newtheorem{lemma}{Lemma}[section]
\def\R{\text{$\mathbb{R}$}}

\def\fis#1{\dot{#1}}
\def\lra{\longrightarrow}

\def\d1#1#2{\frac{d#1}{d#2}}

\def\tpm{T_p M}

\def\p1#1#2{\frac{\partial #1}{\partial #2}}

\def\pr{\parallel}

\def\part{a=t_0 < t_1 < \ldots < t_{k-1} < t_{k}=b}

\def\C{\text{$\mathbb{C}$}}
\def\N{\text{$\mathbb{N}$}}
\def\H{\text{$\mathbb{H}$}}
\def\Z{\text{$\mathbb{Z}$}}

\def\vai{\rightarrow}

\newcommand{\riem}[2]{({#1},{#2})}

\newcommand{\euclideo}{\langle \cdot , \cdot \rangle}

\newcommand{\st}[2]{\textbf{St}({#1},{#2})}
\newcommand{\gr}[2]{\textbf{Gr}({#1},{#2})}
\newcommand{\lin}[2]{\textrm{L}(\R^{#1},{#2})}
\newcommand{\linc}[2]{\textrm{L}(\C^{#1},{#2})}
\newcommand{\id}[1]{\mathrm{Id}_{\R^{#1}}}
\newcommand{\idc}[1]{\mathrm{Id}_{\C^{#1}}}
\newcommand{\stc}[2]{\textbf{St}^{\C}({{#1}},{{#2}})}
\newcommand{\grc}[2]{\textbf{Gr}^{\C}({#1},{#2})}
\begin{document}
\author{Leonardo Biliotti, Mercuri Francesco}
\title{Properly discontinuous actions on Hilbert manifolds}
\address{Universit\`{a} di Parma} \email{leonardo.biliotti@unipr.it}
\address{Universidade Estadual de Campinas (UNICAMP)}
\email{mercuri@ime.unicamp.br}
\thanks{ The authors were partially supported by FIRB 2012 ``Geometria differenziale e teoria geometrica delle functioni'' grant $RBFR12W1AQ003$. The second author was partially supported by Fapesp and CNPq (Brazil)} %
\keywords{Hilbert manifold, homogeneous space, properly discontinuous action}
\subjclass[2000]{58B99;57S25}
 \maketitle
\begin{abstract}
In this article we study properly discontinuous actions on Hilbert manifolds giving new examples of complete Hilbert manifolds with nonnegative, respectively nonpositive, sectional curvature with infinite fundamental group. We also get examples of complete infinite dimensional K\"ahler manifolds with positive holomorphic sectional curvature and infinite fundamental group in contrast with the finite dimensional case and we classify abelian groups acting linearly, isometrically and properly discontinuously on Stiefel manifolds. Finally, we classify homogeneous Hilbert manifolds with constant sectional curvature.
\end{abstract}
\tableofcontents{}
\section{Introduction}
A Hilbert manifold $\riem{M}{\langle \cdot , \cdot \rangle}$ is a smooth manifold modeled on a Hilbert space $\H$ and equipped with an inner product $\langle \cdot , \cdot \rangle (p)$ on any tangent space $\tpm$
depending smoothly on $p$ and defining on $\tpm \cong \H$ a norm equivalent to the original norm of $\H$.

The investigation of global properties in infinite dimensional geometry is harder than in the finite dimensional case essentially because of the lack of local compactness. For example there exist complete Hilbert manifolds with points that cannot be connected by  minimal geodesics  and one can construct on such manifolds finite geodesic segments containing  infinitely many conjugate points \cite{Gr}. A complete description of conjugate points along finite geodesic segment is given in \cite{bpet} and similar questions have been studied in \cite{bi2,Mi3,Mi,Mi2,Mi4}. Moreover, there exist complete Hilbert manifolds such that the exponential map fails to be surjective \cite{at}. However, Ekeland \cite{ek} proved that almost all points can be joined to a prescribed endpoint by a unique minimal geodesic.

Stiefel and Grassmann manifolds have been intensively studied for many authors in different contexts \cite{ba,do,eas,hm,ja,mm}.  We recall that the Stiefel manifold of the $p$ frames in a Hilbert space $\H$, that we denote by $\st{p}{\H}$, can be endowed by a Riemannian metric $\langle \cdot , \cdot \rangle$, induced by the natural embedding of $\st{p}{\H}$ into the Hilbert space $\lin{p}{\H}$ of linear maps of $\R^p$\ in $\H$. The orthogonal group $\mathrm{O}(p)$ acts freely and isometrically on $\st{p}{\H}$ and so the Grassmann manifold of $p$ subspaces of $\H$,\ $\gr{p}{\H}=\st{p}{\H}/\mathrm{O}(p)$,\  can be endowed by a Riemannian metric such that the natural projection $\pi:\st{p}{\H}\lra \gr{p}{\H}$ is a Riemannian submersion \cite{eas,hm}.
In \cite{hm}, using the computation given in \cite{eas}, it has been proved   that, if $\H$\ is separable, any two points in these manifolds can be connected by a minimal geodesic (in section \ref{sectionstiefel} we remove the condition of separability).

In the present paper we study properly discontinuous actions on Stiefel and Grassmannian manifolds applying arguments developed in \cite{bi}. The study of this kind of actions is needed in order to give new examples of Hilbert manifolds with given fundamental group. Our first main result is the following
\begin{thm}\label{main1}
Let $H\cong G\oplus \Z_{p_1^{\alpha_1}}\oplus \cdots \oplus \Z_{p_k^{\alpha_k}}$ where $G$ is a torsionfree group. Let  $\st{p}{\H}$ be the Stiefel manifold of the $p$ orthonormal frames in $\H$, where $\H$ is an infinite dimensional Hilbert space whose Hilbert basis has the same cardinality of $G$.
Then $H$ acts linearly, isometrically and properly discontinuously on $\st {p}{\H}$ if and only if $p_i\neq p_j$ whenever $i\neq j$. Moreover, $G$ acts properly discontinuously on  $\gr{p}{\H}$.
\end{thm}
The above theorem gives a complete classification of the abelian groups acting linearly and properly discontinuously on the Stiefel manifolds.
Since $\st{p}{\H}$ is contractible whenever $\H$ has infinite dimensional (see \cite{ee}), an important conclusion  of Theorem \ref{main1} is the existence of a complete Hilbert manifold with nonconstant sectional curvature
with fundamental group isomorphic to $H$.

One can define on $\st{p}{\H}$ another metric $g$,\ called \emph{canonical metric} \cite{eas}. We shall prove that $\riem{\st{p}{\H}}{g}$ is a complete Hilbert manifold
with nonconstant and nonnegative sectional curvature and  any two points in $\riem{\st{p}{\H}}{g}$  can be connected by a minimal geodesic.
We also investigate complex Stiefel and Grassmann manifolds and in section \ref{section 2}, we prove the following result
\begin{thm}\label{main2}
Let $G$ be a torsionfree group.  Then there exists a complete Hilbert manifold $M$ of nonnegative and nonconstant sectional curvature such that $\pi_1 (M)\cong G  \oplus  \Z_{p_1^{\alpha_1}}\oplus \cdots \oplus \Z_{p_k^{\alpha_k}}$ with $p_i\neq p_j$ whenever $i\neq j$. Moreover, there exists a complete and infinite dimensional K\"ahler manifold $M$ with positive holomorphic sectional curvature whose fundamental group is isomorphic to $G$.
\end{thm}
Note that the last part of the Theorem is in contrast with the finite dimensional case since a finite dimensional K\"ahler manifold with positive holomorphic sectional curvature is compact and simply connected \cite{tsukamoto}.
Finally we classify homogeneous Hilbert manifolds of constant sectional curvature and we give a new example of an infinite dimensional complete Hilbert manifolds of negative constant sectional curvature whose fundamental group is isomorphic to $\Z^k=\underbrace{\Z \oplus \cdots \oplus \Z}_k$ for every $k\in \N$.

The paper is organized as follows. In Section \ref{prelim} we briefly discuss properly discontinuous actions on Hilbert manifolds and the geometry of the Stiefel and Grassmann manifolds.
From Section \ref{proof1} to Section \ref{homogenous} we prove our main results.
\section{Preliminaries}\label{prelim}
\subsection{Properly discontinuous actions}
In this section we will briefly discuss properly discontinuous actions on a Hilbert manifold. Most of the results hold in a more general context.

Let $G$ be an abstract group and let $M$ be a Hilbert manifold. An action of $G$ on $M$ is a map $\rho:G \times M \lra M$ such that: $\rho(g,\cdot)$ is a diffeomorphism,
$\rho(e,m)=m$ and $\rho(g_1 g_2, m)=\rho (g_1,\rho(g_2,m))$ for every $g,g_1, g_2 \in G$ and for every $m\in M$, where $e \in G$ denotes the  neutral element of $G$.
 Throughout this article we always denote  $\mu(g,m)=g(m)$. The subgroup $G_m=\{g\in G:\  g(x)=x\}$ is called the isotropy group of $x$. The orbit throughout $m$ is the set $G(m)=\{ g(m):\, g\in G \}$. One say the $G$ action is \emph{transitive} if $G$ has just one orbit, i.e., for every $p,q \in M$, there exists $g\in M$ such that $g(p)=q$.  The $G$-action is called \emph{effective}, if $g(m)=m$ for every $m\in M$, implies $g=e$. Hence $G$ acts effectively on $M$ if and only if $\bigcap_{x\in M} G_x =\{e\}$; if $G_x=\{e\}$ for every $x\in M$, we say that $G$ acts freely on $M$. Note that any action can be reduced to an effective action. Indeed, $N=\bigcap_{x\in M} G_x$ is a normal subgroup and $G/N$ acts effectively on $M$.

A $G$ action on $M$ is called \emph{properly discontinuous} if the following conditions hold:
\begin{enumerate}
\item for every $m \in M$, there exists an open neighborhood $U$ of $x$ such that $g(U)\cap U=\emptyset$ for every $g\neq e$;
\item for every $y \notin  G(x)$, there exist neighborhoods $U$ and $V$ of $x\in U$ and $y\in V$ respectively such that  $g(U)\cap V=\emptyset$ for every $g\in G$.
\end{enumerate}
The second condition means that the orbit space  $M/G$ is Hausdorff.
Note that a properly discontinuous action is free and a finite group $G$ acts properly discontinuously on a manifold $M$ if and only if it acts freely on $M$.
The following result is well-known \cite{Kn}.
\begin{prop}
Let $G$ be a group acting properly discontinuously on an Hilbert manifold $M$. Then
orbit space $M/G$ admits a differential structure such that $\pi:M \lra M/G$ is a covering map.
\end{prop}

Assume now that $G$ acts freely on $M$. We say that $G$ acts \emph{discontinuously} on $M$ if for every $x\in M$ and every sequence $a_n$ (where $a_n$ are all mutually distinct) then $a_n (x)$ does not converge.
Hence a properly discontinuous action is discontinuous.

We will say that $G$ acts isometrically on $\riem{M}{\euclideo}$ if the transformation $g:M\lra M$ is an isometry of $M$ for every $g\in G$.
The following result is  a useful criteria for properly discontinuous actions \cite{Kn}.
\begin{prop}\label{dp}
If $G$ acts discontinuously and isometrically on $M$, the action is properly discontinuous. In this case $M/G$ admits a Riemannian structure such that the natural projection
$\pi:M\lra M/G$ is a Riemannian covering map. Moreover, if $M$ is a complete Hilbert manifold  so is $M/G$.
\end{prop}
The last part of the above proposition is proved in \cite{bi2}.
\subsection{Infinite dimensional Stiefel and Grassmann manifolds}\label{sectionstiefel}
In this section we will briefly investigate the Riemannian geometry of the Stiefel manifolds with respect to the euclidian metric and with respect to a Riemannian metric which is called the \emph{canonical metric}.
We also study the Riemannian geometry of the Grassmannian manifolds.

Let $\H$ be a Hilbert space and let $p$ be a positive number. $\lin{p}{\H}$ denotes the set of linear maps from $\R^p$ into $\H$. This is a Hilbert space endowed by the following Hilbert product: $\langle x , y \rangle=\mathrm{Tr}(x^t \circ y)$. Here $x^t \in \mathrm{L}(\H, \R^p )$ denotes the transpose with respect to the metric on $\H$ and on $\R^p$ respectively, i.e.
\[
\langle x(v),Z\rangle_{\H} =\langle v, x^t (Z) \rangle_{\R^p},
\]
for any $v\in \R^p$ and for any $Z\in \H$.\ In particular we have the following orthogonal decomposition:
\[
\H=\mathrm{Im}\, x \oplus \mathrm{Ker}\, x^t.
\]
The Stiefel manifold $\st{p}{\H}$  is the set of linear isometric immersion of $\R^p$ into $\H$. Equivalently,  $\st{p}{\H}=\{x\in \lin{p}{\H}:\, x^t \circ x =\id{p}\}$ and so it is a closed smooth submanifold of finite codimension $\frac{1}{2}p(p+1)$ of $\lin{p}{\H}$. The Hilbert product $\euclideo$ on the Hilbert space $\lin{p}{\H}$ induces a Riemannian metric on $\st{p}{\H}$ such that
$\riem{ \st{p}{\H}}{\euclideo}$ is a complete Hilbert manifold.
The tangent space at $Y$ is given by: $T_Y \st{p}{\H}=\{V:\R^p \lra \H:\ Y^t \circ V$ is skew symmetric$\}$.
If $\H$ has finite dimensional, we tacitely assume that $\dim \H \geq 2p$.

If $T:\H \lra \H$ is a linear map, then $T$ induces in a natural way a linear map $\tilde T$  on $\lin{p}{\H}$, by setting
\[
\tilde T (x)=T \circ x.
\]
It is easy to check that if $T:\H \lra \H$ is an isometry, then $\tilde T$ is an isometry of $\lin{p}{\H}$ preserving $\st{p}{\H}$. Therefore $\tilde T$ is also an isometry of $\riem{\st{p}{\H}}{\euclideo}$.
In particular,  if $H$ acts linearly and by isometry on $\H$, then it also acts isometrically  on $\st{p}{\H}$, by setting
\[
\mu:H \times \st{p}{\H} \lra \st{p}{\H}\ \ \ (g,x) \mapsto \tilde g (x).
\]
Hence the orthogonal group $\mathrm{O}(\H)$ of $\H$ acts isometrically and transitively on $\st{p}{\H}$.
More generally, if $T:V \lra W$ is a linear isometric immersion, then the  map
\[
\tilde T:\st{p}{V}\lra \st{p}{W},\ \ x\mapsto T \circ x.
\]
is smooth and the following result holds \cite{hm}.
\begin{lemma}\label{sg1}
$\tilde T:\st{p}{V}\lra \st{p}{W}$ is an isometric embedding and $\tilde T (\st{p}{V} )$ is a totally geodesic submanifold of $\riem{\st{p}{W}}{\euclideo}$.
\end{lemma}
In \cite{hm} the authors proved that $\riem{\st{p}{\H}}{\euclideo}$ is \emph{Hopf-Rinow} whenever $\H$ is separable. This means that any two points can be connected by a minimizing geodesic. The next result shows that the hypothesis of separability is not necessary.
\begin{prop}\label{hopf-rinow}
Let $\H$ be a Hilbert space. Then $\riem{\st{p}{\H}}{\euclideo}$ is Hopf-Rinow.
\end{prop}
\begin{proof}
Let $x,y\in \st{p}{\H}$. By Ekeland Theorem \cite{ek}, there exists a sequence $y_n$ converging to $y$ such that: there exists a unique minimizing geodesic joining $x$ and $y_n$; $\lim_{n\mapsto \infty} d(x,y_n) = d(x,y)$. Now, any geodesic joining $x$ and $y_n$ belongs to a finite dimensional Stiefel manifolds \cite{eas,hm}. Hence there exits a separable Hilbert space $W$ such that: $x,y,y_n \in \st{p}{W}$ as well as the minimizing  geodesic joining $x$ and $y_n$. By \cite{hm} there exists a minimizing geodesic joining $x$ and $y$. Since $\st{p}{W}$ is totally geodesic in $\st{p}{\H}$ then the unique minimizing geodesic joining $x$ and $y_n$ is also the unique minimizing geodesic in $\st{p}{W}$ as well. Now, keeping in mind that $d(x,y_n)$ is the same either $\st{p}{\H}$ or  $\st{p}{W}$,   we get that the minimizing geodesic in $\st{p}{W}$ is a minimizing geodesic in $\st{p}{\H}$ concluding the proof.
\end{proof}
The orthogonal group $\mathrm{O}(p)$ acts freely and isometrically on $\riem{\st{p}{\H}}{\euclideo}$ as follows
\[
\mathrm{O}(p) \times \st{p}{\H} \lra \st{p}{\H},\ (A,x) \mapsto x \circ A^t.
\]
The quotient space $\gr{p}{\H}=\st{p}{\H}/ \mathrm{O}(p)$ is the Grassmann manifold of the $p$-dimensional subspaces of $\H$. It can be endowed by a Riemannian structure, that we also denote by $\langle \cdot, \cdot \rangle$, such that the natural projection $\pi:\st{p}{\H} \lra \gr{p}{\H}$ is a Riemannian submersion.
Since the $\mathrm{O}(\H)$ action on $\st{p}{\H}$  commutes with the $\mathrm{O}(p)$ action, then $\mathrm{O}(\H )$ acts isometrically and transitively also on $\gr{p}{\H}$. In particular $\gr{p}{\H}$\ is {\em homogeneous}.
\begin{prop}\label{grassmannian-complete}
Let $\H$ be a Hilbert space. Then  $\riem{\gr{p}{\H}}{\euclideo}$  is a complete Hilbert manifold satisfying the Hopf-Rinow theorem.
\end{prop}
\begin{proof} Let $p \in \gr{p}{\H}$. Consider the geodesic ball \[ B(p,\epsilon) = \{ q \in \gr{p}{\H}: d(p,q) < \epsilon\}\] where
$d$ is the distance function defined by the Riemannian metric. Then $\overline{B(p,\epsilon)}$ is a complete metric space, for $\epsilon$\ small (see \cite{ek}).\ Let $\{x_n\}$\ be a Cauchy sequence in $\gr{p}{\H}$.\ Then there exist $n_0$\ such that $x_n \in B(x_{n_o},\epsilon)$\ for every $n\geq  n_o$. Since $\gr{p}{\H}$ is homogeneous, $\overline{B(x_{n_o},\epsilon)}$ is a complete metric space as well.  Hence the sequence admits a convergent subsequence and, being Cauchy, it converges. We can apply now the results in \cite{hm} and the argument in Proposition \ref{hopf-rinow} to get the result.
\end{proof}
\begin{remark}
We point out that also Theorem 3  in \cite{hm} works for the Grassmann manifold $\riem{\gr{p}{\H}}{\euclideo}$.
\end{remark}

We can also consider the so called \emph{canonical metric} on the Stiefel manifold. For every $Y\in \st{p}{\H}$ we define
\[
g_{Y} (V,W)=\mathrm{Tr}(V^t ( \textrm{Id} -\frac{1}{2}YY^t)W).
\]
on $T_Y \st{p}{\H}$. Note that $L_Y= \textrm{Id}-\frac{1}{2}YY^t$ is linear endomorphism of $\H$ depending smoothly on $Y$. We check that $L_Y$ is an invertible operator whenever $Y\in \st{p}{\H}$.
Indeed, keeping in mind the  orthogonal splitting
\[
\H=\mathrm{Im}\, Y \oplus \mathrm{Ker}\, Y^t,
\]
we get $(L_Y)_{|_{\mathrm{Im}\, Y}}=\frac{1}{2} \textrm{Id}_{|_{\mathrm{Im}\, Y}}$ and $(L_Y )_{|_{\mathrm{Ker}\, Y^t}}=\textrm{Id}_{|_{ \mathrm{Ker} Y^t}}$ and so
$L_Y$ is an invertible operators whose norm satisfies $\frac{1}{2} \leq \parallel L_Y \parallel \leq 1$. Note also that $L_Y$ is a symmetric and positive definite operator of $\H$.
We know that the isomorphism $L_Y$ induces endomorphism $\tilde L_Y$ of $\lin{p}{\H}$, by setting
\[
\tilde L_Y: \lin{p}{\H} \lra \lin{p}{\H},\ \ V \mapsto L_Y \circ V.
\]
This is a continuous and invertible endomorphism whose inverse in given by: $(\tilde L_Y )^{-1}=\widetilde{ (L_Y)^{-1}}$. Since $g_{Y}=\langle \cdot , \tilde L_Y \cdot \rangle$, this proves that $g_{Y}$ is a nondegenerate bilinear form. We shall prove that $g_{Y}$ defines a Riemannian metric on $\st{p}{\H}$. Firstly, we check that $g_{Y}$ is a symmetric bilinear form.
\[
\begin{split}
g_{Y}(V,W)&=\mathrm{Tr}(V^t \tilde L_Y (W))\\
            &=\mathrm{Tr}(V^t (L_Y \circ W)) \\
            &=\mathrm{Tr}((L_Y \circ V)^t W) \ \mathrm{since}\ L_Y\ \mathrm{is\ symmetric}\\
            &=\mathrm{Tr}(W^t (L_Y \circ V)) \\
            &=g_{Y}(W,V).
\end{split}
\]
Now, we shall prove that $g_{Y}$ is a scalar product on $T_Y \st{p}{\H}$ which defines a norm equivalent to $\euclideo$. Indeed,
\[
\begin{split}
g_{Y}(V,V)&=\mathrm{Tr}(V^t V)-\frac{1}{2}\mathrm{Tr}(V^t Y Y^t V) \\
            &=\mathrm{Tr}(V^t V)-\frac{1}{2}\mathrm{Tr}((Y^t V)^t Y^t V)\\
            &\leq \langle V,V \rangle.
\end{split}
\]
where the last inequality follows from the fact that $\mathrm{Tr}((Y^t V)^t Y^t V)\geq 0$.
Moreover, if $e_1, \ldots, e_p$ is an orthonormal basis of $\R^p$, and we denote by $P_{\mathrm{Im}\, Y}$ the orthogonal projection on $\mathrm{Im}\, Y$, then
\[
\begin{split}
\mathrm{Tr}((Y^t V)^t Y^t V)&=\sum_{i=1}^p \langle Y^t V(e_i), Y^t V (e_i ) \rangle  \\
            &=\sum_{i=1}^p \langle Y^t P_{\mathrm{Im}\, Y} (V(e_i)), Y^t P_{\mathrm{Im}\, Y} (V (e_i ) )\rangle \\
            &\leq \sum_{i=1}^p \langle V(e_i),V(e_i) \rangle \\
            &=\langle V,V \rangle,
\end{split}
\]
where the inequality follows from the fact that $Y$ is a linear isometric immersion of $\R^p$ in $\H$. Therefore
\begin{equation}\label{e1}
\frac{1}{2} \langle V , V \rangle \leq g_{Y}(V,V) \leq  \langle V , V \rangle,
\end{equation}
which also means that $\riem{\st{p}{\H}}{g}$ is a complete Hilbert manifold since $\riem{\st{p}{\H}}{\euclideo}$ is. Moreover the following result holds.
\begin{lemma}\label{asc}
If  $T:\H \lra \H$ is an isometry, then the induced map $\tilde{T}$ is an isometry of $\riem{\st{p}{\H}}{g}$. Therefore if $G$ acts isometrically on $\H$, then  $G$ acts isometrically on $ \riem{\st{p}{\H}}{g}$ as follows:
\[
G \times  \st{p}{\H} \lra \st{p}{\H},\ \  (h,x)\mapsto \tilde h (x).
\]
Therefore $\mathrm{O}(\H)$ acts isometrically and transitively on $\riem{\st{p}{\H}}{g}$
\end{lemma}
\begin{proof}
We only check the first part of the Lemma.
\[
\begin{split}
g_{T\circ Y}(T \circ V, T \circ W )&=\mathrm{Tr}((V^t \circ T^t)( \textrm{Id}-\frac{1}{2}(T \circ Y)( Y^t \circ T^t ))T\circ W)\\
                              &=\mathrm{Tr}(V^t ( \textrm{Id}-\frac{1}{2}Y Y^t )W)\\
                              &=g_{Y}(V,W).\\
\end{split}
\]
\end{proof}
The Riemannian metric $g$ on $\st{p}{\H}$ it is called \emph{canonical metric} since if $\H$ has finite dimension $g$ is, up to a constant, the Riemannian metric such that the map
\[
\mathrm{O}({\H}) \mapsto \st{p}{\H}\, \ \ A \mapsto A \circ Y
\]
is a Riemannian submersion, where $\mathrm{O}(\H)$ is endowed by a bi-invariant metric \cite{ce,eas}. Hence by the O'Neil formula, see \cite{ce},
the sectional curvature of $\riem{\st{p}{\H}}{g}$ is nonnegative and nonconstant whenever $p\geq 2$ and $\H$ is finite dimensional.
The $\mathrm{O}(p)$ action is also an isometric action with respect to the \emph{canonical metric} so it induces a Riemannian structure on $\gr{p}{\H}$ such that $\pi:\riem{\st{p}{\H}}{g} \lra \gr{p}{\H}$  is a Riemannian submersion. We point out that the metric induced on $\gr{p}{\H}$ is the same of the metric induced by the euclidian metric. Indeed, $T_Y \mathrm O (p) (Y) =\{Y\circ A:\, A$ is skew-symmetric $\}$
and so
$g_{Y}(V,Y\circ A)=0$ for every $A$ skew symmetric if and only if
\[
\mathrm{Tr}(V^t (\textrm{Id}-\frac{1}{2}YY^t)Y\circ A)=\frac{1}{2}\mathrm{Tr}((Y^t \circ V)^t A)=0,
\]
and so if and only if $V(\R^p)\subset \mathrm{Ker}\, Y^t$ since $Y^t \circ V$ is skew symmetric. Since $g_{Y}(V^t,Y \circ A)=\frac{1}{2}\langle V, Y\circ A \rangle$, we get
\[
\begin{split}
\Bigl(T_Y  \mathrm O(p) (Y) \Bigr)^{\perp_{\euclideo}}&=\Bigl( T_Y \mathrm O (p) (Y)  \Bigr)^{\perp_g}\\ &=\{V:\R^p \lra \H:\ V(\R^p )\subset \mathrm{Ker}\, Y^t \}
\end{split}
\]
and $g_Y$ restricted to $\Bigl(T_Y \mathrm O (p) (Y) \Bigr)^{\perp_{\euclideo}}=\Bigl(T_Y \mathrm O (p)(Y) \Bigr)^{\perp_{g}}$ coincides with $\euclideo$ restricted to $\Bigl(T_Y \mathrm O (p)(Y) \Bigr)^{\perp_{\euclideo}}=\Bigl(T_Y  \mathrm O (p)(Y)\Bigr)^{\perp_{g}}$.

We will denotes again by $g$ such Riemannian metric on $\gr{p}{\H} $. As before $\mathrm{O}(\H)$ acts transitively and isometrically on $\gr{p}{\H}$ and so $\riem{\gr{p}{\H}}{g}$ is a complete Hilbert manifold. The following Lemma is just Lemma \ref{sg1} for the canonical metric and the proof is similar to one given in \cite{hm}. For sake of completeness we give the sketch of the proof.
\begin{lemma}\label{l3}
Let $L:W \lra V$ be an isometric linear immersion. Then $L$ induces isometric immersions,\ denoted by the same symbol,
\[
\tilde L:\st{p}{W} \lra \st{p}{V},\ \ x \mapsto L \circ x,
\]

\[
\tilde L:\gr{p}{W} \lra \gr{p}{V},\ \ \sigma \mapsto L (\sigma).
\]
Moreover, $\tilde L (\st{p}{W}) \hookrightarrow \st{p}{V}$ and $\tilde L (\gr{p}{W}) \hookrightarrow \gr{p}{V}$\ are totally geodesic.
\end{lemma}
\begin{proof}
Let $V,W\in T_Y \st{p}{W}$. Then
\[
\begin{split}
g_{(L\circ Y)}(L \circ V,L \circ W )&=\mathrm{Tr}((V^t \circ L^t)( \textrm{Id}-\frac{1}{2} (L \circ Y) (Y^t \circ L^t))L \circ W)\\
                              &=\mathrm{Tr}(V^t ( \textrm{Id}-\frac{1}{2}Y Y^t )W)\\
                              &=g_Y(V,W)\\
\end{split}
\]
We identify $W$\ with $L(W)$ and consider the splitting $V=W\oplus {W}^\perp$.\ Let  $T:V \lra V$ be the following isometry: $T_{|_{W}}=Id_{W}$ and $T_{|_{W^\perp}}=-Id_{W^\perp}$. It is easy to check that
\[
\{x\in \st{p}{V}:\, T(x)=x\}=\st{p}{W},
\]
and so $\st{p}{W}$\ is totally geodesic \cite{Kl,Kn}.

The map $\tilde L$ is $\mathrm{O}(p)$-equivariant and so it induces an isometric immersion of $\gr{p}{V}$ into $\gr{p}{W}$
and the fixed points set with respect to $\mathrm{O}(W^\perp )$ embedded in $\mathrm{O}(V)$ as  $A\mapsto \textrm{Id} \oplus A$, coincides with $\gr{p}{W}$ and so $\gr{p}{W}$ is totally geodesic.
\end{proof}
\begin{cor}\label{sc1}
$\riem{\st{p}{\H}}{g}$\ and $\riem{\gr{p}{\H}}{g}$, are complete Hilbert manifolds with nonnegative sectional curvature satisfying the Hopf-Rinow Theorem.
\end{cor}
\begin{proof}
We shall prove the result for $(\st{p}{\H},{g})$. The proof for $\riem{\gr{p}{\H}}{g}$ is similar. By Theorem 2.1 in \cite{eas}, we get that any geodesic is contained in a finite dimensional Stiefel manifold. Hence, using the arguments developed in \cite{hm}, it can be proved that $\st{p}{\H}$ is Hopf-Rinow whenever $\H$ is separable. Moreover, Proposition \ref{hopf-rinow} works in this context and so $\riem{\st{p}{\H}}{g}$ is Hopf-Rinow.

Let $Y\in \st{p}{\H}$ and let $V,W\in T_Y \st{p}{\H}$. It is easy to check that  $\mathbb V= \mathrm{span} \bigl( Y(\R^p ), V (\R^p) ), W  (\R^p) ) \bigr)$ is a finite dimensional subspace such that
$Y\in \st{p}{\mathbb V}$ and $V,W\in T_Y \st{p}{\mathbb V}$. Hence by Lemma \ref{l3}, the sectional curvature $K_{Y}(V,W)$\ is nonnegative and the result follows.
\end{proof}
\begin{remark}
Theorem 3 in \cite{hm} can be proved in our context.
\end{remark}
\begin{remark}
As in the finite dimensional case, $\gr{p}{\H}$ is a symmetric Hilbert  manifold.
Let $W\in \gr{p}{\H}$ and let $T:\H \lra \H$ be the isometry of $\H$ such that $T_{|_{W }}=Id_{W}$ and $T_{|_{W^\perp}}=-Id_{W^\perp}$. Let  $\sigma_W$ be the isometry induced on $\gr{p}{\H}$ by $\tilde T$.
It is an involution of $\gr{p}{\H}$ such that $\sigma_W (W)=W$.
We shall prove that $(\mathrm d \sigma )_W =-\mathrm{Id}_{|_{ {T_{W} \gr{p}{\H}}}}$ and so $\gr{p}{\H}$ is a symmetric space.

Let  $Y\in \st{p}{\H}$ such that $Y(\R^p)=W$. As we saw before,
\[T_Y
\st{p}{\H}=\Bigl( T_Y \mathrm{O}(p)(Y) \Bigr)^\perp  \oplus\, T_Y  \mathrm{O}(p)(Y),
\]
where
$
\Bigl( T_Y \mathrm O (p)(Y) \Bigr)^\perp=\{V:\R^p \lra \H:\ V(\R^p)\subset W^\perp \}
$
and so the following diagram:
\[
\xymatrix{
\Bigl(T_Y  \mathrm O (p) (Y)   \Bigr)^\perp \,     \ar[d]_{(\mathrm d \pi )_Y}  \ar[r]^{-\mathrm{Id}} & \Bigl(T_Y \mathrm O (p) (Y)  \Bigr)^\perp  \ar[d]^{(\mathrm d \pi )_Y} \\
T_{W} \gr{p}{\H}   \ar[r]^{(\mathrm d \sigma_W )_W } &  T_{W} \gr{p}{\H}. \\
}
\]
is commutative. Therefore $(\mathrm d \sigma_W )_W = -\mathrm{Id}_{|_{ {T_{W} \gr{p}{\H}}}}$.
\end{remark}
\subsection{Complex Stiefel and Grassmannian manifolds}\label{complex-grassmannian}
In this subsection we briefly discuss complex Stiefel and Grassmannian manifolds.
Let $\H$ be a complex Hilbert space and let $\linc{p}{\H}$ the set of complex linear maps which is a complex Hilbert space with respect to the following Hermitian Hilbert product: $h(X,Y)=\mathrm{Tr}(Y^* X)$.
Moreover, $h(\cdot, \cdot)=\euclideo - i\omega (\cdot,\cdot)$, $\euclideo$ defines a real Hilbert structure on $\linc{p}{\H}$ and $\omega$ is a symplectic form, i.e. $\omega:\linc{p}{\H} \times \linc{p}{\H} \lra \R$ is skew-symmetric and the natural map associated to $\omega$, i.e., $\tilde \omega: \linc{p}{\H} \lra \linc{p}{\H}^*$ defined by setting $\tilde \omega (L)=\omega(L,\cdot)$, is an isomorphism. If we denote by $J$ the multiplication by $i$ in $\H$, $J$ defines an almost complex structure on $\linc{p}{\H}$, that we also denote by $J$, such that $\omega(\cdot,\cdot)=\langle J \cdot,\cdot \rangle$. Hence $\omega$ is a K\"ahler form on $\linc{p}{\H}$ \cite{Kn2}.
The complex Stiefel manifold can be viewed as
$\stc{p}{\H}=\{x:\C^p \lra \H:\, x^* \circ x=\idc{p} \}$ and its tangent space is given by $T_Y \stc{p}{\H}=\{V\in \linc{p}{\H}:\ Y^* \circ V$ is skew-hermitian$\}$.
If we restrict $\euclideo$ on $\stc{p}{\H}$, then $\riem{\stc{p}{\H}}{\euclideo}$ is a complete Hilbert manifold.

The Lie group $\mathrm U (p)$ acts isometrically on $\linc{p}{\H}$ by setting $A \cdot \phi= \phi \circ A^*$ and it preserves $\stc{p}{\H}$ on which it acts freely. Moreover, one can check that
\[
\begin{array}{l}
\Bigl( T_Y \st{p}{\H} \Bigr)^{\perp_{h}}=J\Bigl( T_Y \mathrm U(p) (Y) \Bigr) \\
T_Y \stc{p}{\H} \cap T_Y \stc{p}{\H}^{\perp_{\omega}}=T_Y U(p) (Y) , \\
T_Y \stc{p}{\H} =  T_Y \mathrm{U} (p) (Y) \oplus \Bigl( T_Y \mathrm{U}(p) (Y) \Bigr)^{\perp_{h}},
\end{array}
\]
and $(\mathrm d \pi)_Y :  \Bigl( T_Y \mathrm{U}(p) (Y) \Bigr)^{\perp_{h}} \lra T_Y \grc{p}{\C}$ is an isomorphism for every $Y\in \stc{p}{\H}$, where $\pi:\stc{p}{\H} \lra \grc{p}{\H}$ is the natural projection. Hence $\omega$ and $J$ on $\linc{p}{\H}$ induce a K\"ahler structure on $\grc{p}{\H}$. Note also that $\mathrm{U}(\H)$ commutes with the $\mathrm{U}(p)$ actions and so it acts by holomorphic isometries and transitively on $\grc{p}{\H}$. Therefore $\grc{p}{\H}$ is a complete K\"ahler manifold and if a group $G$ acts isometrically on $\H$, then it acts isometrically on $\linc{p}{\H}$ by setting $(g,Y)\mapsto g\circ Y$, and so it acts by holomorphic isometries on $\grc{p}{\H}$. The following Lemma can be proved as Lemma \ref{l3} and Corollary \ref{sc1} keeping in mind that $\grc{p}{\H}$ has positive holomorphic sectional curvature whenever $\H$ is finite dimensional. Indeed,
$\frac{2}{p}\leq K_Y (X,J(X))\leq 2$ for every $Y\in \grc{p}{\H}$ and for every $X\in T_Y \grc{p}{\H}$ (see \cite{wong}).
\begin{lemma}\label{hc}
Let $L:W \lra V$ be an isometric linear immersion between complex Hilbert space.  Then $L$ induces an isometric immersion
\[
\tilde L:\grc{p}{W} \lra \grc{p}{V},\ \ \sigma \mapsto L (\sigma).
\]
$\tilde L (\grc{p}{W}) \hookrightarrow \grc{p}{V}$ is totally geodesic. Moreover $\grc{p}{\H}$ is a complete K\"ahler Hilbert manifold with  positive holomorphic sectional curvature.
\end{lemma}
We conclude this section pointing out that $\grc{p}{\H}$ is simply connected. Indeed, let $\gamma:[0,1]\lra \grc{p}{\H}$ be a loop. Then there exists an $\epsilon >0$ such that for every $t\in[0,1]$, $\exp_{\gamma(t)}$ is a diffemorphism from the ball of radius $\epsilon$ in $T_{\gamma(t)} \grc{p}{\H}$ onto the geodesic ball of $\gamma(t)$ of radius $\epsilon$. Then there exists a broken closed geodesic $c$ starting from $\gamma(0)$, homotopically  equivalent to $\gamma$, with the following property: there exists a partition $0 < t_1 < \ldots < t_n < 1$ such that $c_{|_{[t_i,t_{i+1}]}}$ is the unique minimal geodesic joining $c(t_i)$ and $c(t_{i+1})$ with length lesser than $\epsilon$ for $i=0,\ldots,n-1$. Hence if we take $W=\mathrm{span}(c(0),\ldots,c(n), \fis c (0), \ldots , \fis c (n))$, by the above Lemma it is easy to prove $c(t)\in \grc{p}{W}$. This means that $c$, and so $\gamma$, is homotopically equivalent to a constant loop since $\grc{p}{W}$ is simply connected whenever $W$ has finite dimension \cite{huckleberry-introduction-DMV}.
\section{Proof of Theorem \ref{main1}}\label{proof1}
We know that if $H$ acts linearly and by isometry on $\H$, then it also acts isometrically  on $\st{p}{\H}$ by setting
\[
\mu:H \times \st{p}{\H} \lra \st{p}{\H}\ \ \ (g,x) \mapsto \tilde g (x).
\]
The following fact holds.
\begin{lemma}\label{sp1}
If $H$ acts isometrically and properly discontinuously on then unit sphere of $\H$, then it acts isometrically and properly discontinuously on $\st{p}{\H}$.
\end{lemma}
\begin{proof}
Since $H$ acts freely on the unit sphere of $\H$, then it also acts freely on $\st{p}{\H}$. We prove that $H$ acts properly discontinuously on $\st{p}{\H}$ applying Proposition \ref{dp}.

Let $a_n$ be a sequence in $H$  and let $x\in \st{p}{\H}$. If $a_n (x)$ converged to $y\in \st{p}{\H}$, then $(a_n x )(v)$ would converge to $y(v)$ for every $v\in \R^p$. However, if $\parallel v \parallel =1$, then both $(a_n x)(v)$ and $y(v)$ belong to the unit sphere of $\H$ and so this is an absurd. Therefore $H$ acts properly discontinuously on $\st{p}{\H}$.
\end{proof}
Now we shall prove a similar result for the Grassmannian manifolds.
\begin{lemma}\label{gp1}
Let $G$ be a torsionfree group acting isometrically and properly discontinuously on then unit sphere of $\H$.
Then $G$ acts isometrically and properly discontinuously on $\riem{\gr{p}{\H}}{\euclideo}$.
\end{lemma}
\begin{proof}
Firstly we prove that $G$ acts freely on $\gr{p}{\H}$. Indeed, if $g (\sigma) =\sigma$, then $\Z$ would act properly discontinuously on the unit sphere of $\sigma$, by setting $(n,v)\mapsto g^n (v)$, which is not possible since $\sigma$ has finite dimensional. Let $\sigma\in \gr{p}{\H}$ e let $a_n$ be a sequence in $G$ and suppose $a_n ( \sigma)$ converges to $\sigma_o$.
Let $x\in \st{p}{\H}$ such that $x(\R^p)=\sigma$ and let $y\in \st{p}{\H}$ such that $y(\R^p)=\sigma_o$. Since $\pi:\st{p}{\H} \lra \gr{p}{\H}$ is a fiber bundle, there exists a sequence $k_n \in \mathrm{O}(p)$
such that $a_n \circ x \circ k_n$ converges to  $y\in \st{p}{\H}$. We may assume that $k_n \mapsto k$. Therefore $a_n (x\circ k)$ converges to $y$ that is a contradiction by Lemma \ref{sp1}.
\end{proof}
Let $G$ be a set. It is well known, see \cite{Ru}, that
$$l_2(G) = \{x: G \vai \R : \sum_{g\in G}[x(g)]^2 < \infty \},$$
is a Hilbert space  and a Hilbert
basis is given  by the functions $e_h(g) = \delta_{hg},\, h, g \in G$.
Moreover, for every bijective map  $\phi:G \lra G$, if
$x \in l_2(G)$ then $x\ \circ \phi \in l_2 (G)$ and
$\pr x \pr = \pr x \circ  \phi \pr$.

If $G$ is a group, then $G$ acts isometrically on $l_2 (G)$ by setting:
$$\mu : G \times l_2(G) \vai l_2(G)\ \ \mu(g,x) =gx= x\circ R_g ,$$
where $R_g$ is the right translation.

The following result is proved in \cite{bi}.
\begin{thm}\label{miotesi}
Let $G$ be a torsionfree group.  If $H\cong G\oplus \Z_{p_1^{\alpha_1}}\oplus \cdots \oplus \Z_{p_k^{\alpha_k}}$,  then $H$ acts properly discontinuously on the unit sphere of $l_2 (G)$ if and only if $p_i\neq p_j$ whenever $i\neq j$.
\end{thm}

Hence Theorem \ref{main1} follows from  Theorem \ref{miotesi}, Lemmata \ref{sp1} and \ref{gp1} and Proposition \ref{dp}.
\section{Proof of Theorem \ref{main2}}\label{section 2}
By Proposition \ref{asc}, if a group acts isometrically on $\H$, it also acts isometrically on $\riem{\st{p}{\H}}{g}$.
By Equation (\ref{e1}), we get
\[
\frac{1}{2} d_{\euclideo} (\cdot, \cdot) \leq d_{g} (\cdot, \cdot ) \leq d_{\euclideo} (\cdot, \cdot).
\]
Therefore, by Lemma \ref{sp1}, if $G$ acts properly discontinuously on $\H$, then it acts properly discontinuously on $\riem{\st{p}{\H}}{g}$. Since $\riem{\st{p}{\H}}{g}$ is  contractible (see \cite{ee}) and has
 nonnegative sectional curvature, then applying Propositions \ref{dp}, we have  a complete Hilbert manifolds with nonnegative
and nonconstant sectional curvature whose fundamental group is isomorphic to
$G\oplus \Z_{p_1^{\alpha_1}}\oplus \cdots \oplus \Z_{p_k^{\alpha_k}}$ where $G$ is a torsionfree group and $p_i\neq p_j$ whenever $i\neq j$. This proves the first part of Theorem \ref{main2}.

Let $G$ be a set and consider
$$l^\C_2(G) = \{x: G \vai \C : \sum_{g\in G} | x(g) |^2 < \infty \}.$$
Then $l^\C_2(G)$\ is a complex Hilbert space  and a Hilbert
basis is given  by the functions $e_h(g) = \delta_{hg},\, h, g \in G$.
Moreover,  if $G$ is a group, then $G$ acts isometrically on $l^\C_2 (G)$ by setting:
$$\mu : G \times l^\C_2(G) \vai l^\C_2(G)\ \ \mu(g,x) =gx= x\circ R_g .$$

If $G$ is a torsionfree group, it can be proved as in \cite{bi}, that $G$ acts properly discontinuously on the unit sphere of $l^\C_2 (G)$ and so by the arguments developed in Section \ref{complex-grassmannian} $G$ acts by holomorphic isometry on $\grc{p}{\H}$. Moreover, one can check that Lemma \ref{gp1} holds in this context. Summing up we have proved that $G$ acts  by holomorphic isometries and properly discontinuously on $\grc{p}{l^\C_2 (G)}$ and the quotient $\grc{p}{l^\C_2 (G)} /G$ inherits the structure of a K\"ahler manifold such that the covering map $\pi:\grc{p}{l^\C_2 (G)}  \lra \grc{p}{l^\C_2 (G)}/G$
is holomorphic. Since  $\grc{p}{l^\C_2 (G)}$ simply connected,
by Lemma \ref{hc} and Proposition \ref{dp}, we get the last part of  Theorem \ref{main2}.
\begin{remark}
The results proved in Section \ref{proof1} and Section \ref{section 2} remain valid when the Grassmannian manifolds are replaced by the oriented Grassmannian manifolds, i.e., $\mathbf{Gr}_{+} (p,\H)=\st{p}{\H}/\mathrm{SO}(p)$.
\end{remark}
\section{Homogeneous Hilbert manifolds of constant sectional curvature}\label{homogenous}
Let $(M,g)$ be a complete Riemannian manifold of constant sectional curvature $K$. It is well-known, see \cite{La,wolf}, that $M$ is isometric to $\tilde M/\Gamma$,  where $\tilde{M}$ is the complete simply connected Riemannian manifold with constant sectional curvature $K$, $\Gamma$ is a linear group acting isometrically and properly discontinuously on $\tilde M$ and the natural map $\pi:\tilde M \lra \tilde M/\Gamma$ is a Riemannian covering map. In this section we investigate homogeneous Riemannian manifold of infinite dimensional of constant sectional curvature. Our main result is the extension of a result of Wolf \cite{wolf1,wolf} in the infinite dimensional context.
We begin with the following Lemma that is well-known in the finite dimensional context.
\begin{lemma}\label{mainlemma}
Let $\tilde M$ be a complete and simply connected Hilbert manifold of constant curvature and let $\Gamma$ be  linear group acting isometrically and properly discontinuously on $\tilde M$. If $\tilde M/\Gamma$ is homogeneous, then the centralizer of $\Gamma$ in $\mathrm{Iso}(\tilde M)$, acts transitively on $\tilde M$.
\end{lemma}
\begin{proof}
If $g: \tilde M /\Gamma \lra \tilde M /\Gamma$ is an isometry, then there exists an isometry $\tilde g:\tilde M \lra \tilde M$ centralizing $\Gamma$ such that the following diagram:
\[
\xymatrix{
\tilde M \,     \ar[d]_{\pi}  \ar[r]^{\tilde g} &  \tilde M \ar[d]^{\pi} \\
\tilde M /\Gamma   \ar[r]^{g} & \tilde M / \Gamma . \\
}
\]
is commutative. The vice-versa holds as well. Then $\mathrm{Iso}(\tilde M/\Gamma) \cong \mathrm{Z}(\Gamma)$ the centralizer in $\Gamma$ in $\mathrm{Iso}(\tilde M)$.
Hence $\mathrm{Z}(\Gamma)$ is a closed topological group acting transitively on $\tilde M /\Gamma$. Since $\pi$ is a local homeomorphism and it is also $\mathrm{Z}(\Gamma)$-equivariant,
any orbit of $\mathrm{Z}(\Gamma)$ in $\tilde M$ contains an open subset of $\tilde M$ and so it is itself a $\mathrm{Z}(\Gamma)$  invariant open subset of $\tilde M$. This also means that the complement of any orbit is open as well. Therefore, since $\tilde M$ is connected, $\tilde M$ is $\mathrm{Z}(\Gamma)$ homogeneous.
\end{proof}
\begin{defin}
Let $f: M \lra  M$ be an isometry. The \emph{displacement function} $\delta_f$ is defined as
$$\delta_f (x)=d(x,f(x)).$$
We will say that $f$ is a \emph{Clifford translation} if $\delta_f$ is constant.
\end{defin}
\begin{lemma}\label{geodesic}
Let $f:M \lra M$ be a Clifford translation of a complete Hilbert manifold of constant sectional curvature. Then there exists a geodesic $\gamma$ which is invariant under $f$.
\end{lemma}
\begin{proof}
Let $p\in M$ and let $\gamma:[0,1]\lra M$ be the minimal geodesic joining $p$ and $f(p)$. This geodesic exists since $M$ is Hopf-Rinow \cite{La}.
Since $f$ is a Clifford translation, then
\[
\begin{split}
d(\gamma(t),f(\gamma(t)))&=d(p,f(p)) \\
                        &=d(p,\gamma(t))+d(\gamma(t),f(p))\\
                        &=d(\gamma(t),f(p))+d(f(p),f(\gamma(t)))\\
\end{split}
\]
and so the curve formed by $\gamma$ and $f(\gamma(t))$  is a geodesic and the result follows.
\end{proof}
\begin{remark}
The above Lemma works for any Hilbert manifold satisfying the Hopf-Rinow Theorem.
\end{remark}
As in the finite dimensional, the following Lemmata hold \cite{wolf}.
\begin{lemma}\label{lh1}
Suppose that  $\tilde M/\Gamma$ is homogeneous and  let $g\in \Gamma$. Then $g$ is a Clifford translation.
\end{lemma}
\begin{lemma}\label{lh2}
The Clifford translations of $\H$ are just the ordinary translations. The only Clifford translation of hyperbolic space is the identity.
\end{lemma}
We can now classify complete Hilbert manifolds of constant sectional curvature.
Naturally we can assume that the sectional curvature is $0$, or $\pm1$.
\begin{thm}
Let $M$ be a homogeneous Hilbert manifold of constant sectional curvature $K\leq 0$. If $K= -1$, then $ M$ is isometric to the hyperbolic space. If $K=0$, then $M$ is isometric to  $\H /\Gamma$, where $\H$ is a Hilbert space, $\{v_i \in \H:\  i\in I\}$ is a family of linearly independent vectors of $\H$  and   $\Gamma=\mathrm{span}_{\Z}(v_i:\, i\in I )$.
\end{thm}
\begin{proof}
By Lemmata \ref{lh1} and \ref{lh2}, if $M/\Gamma$ is homogeneous and $K<0$, then $\Gamma=\{e\}$ and so $M$ is isometric to the hyperbolic space. If $K=0$, then $\Gamma$ must contains just ordinary translation. Hence we may write $\Gamma=\{v_i\in \H, i\in I\}$. Let  $J\subset I$ be a finite subset of $I$, and let $\sigma =\mathrm{span} (v_k:\ k\in J)$. Let $\tilde J=\{v\in \Gamma:\ v\in \sigma \}$. Then $\tilde J$ acts freely and properly discontinuously on  the finite dimensional subspace $\sigma$. Therefore, see \cite{wolf}, there exists $v_1',\ldots,v_k' \in \tilde J$ such that $\tilde J=\mathrm{span}_\Z (v_1',\ldots , v_k')$  and the theorem follows.
\end{proof}
\begin{thm}\label{hp}
Let $S(\H)$ be the unit sphere of an infinite dimensional Hilbert space $\H$.
Let $\Gamma$ be a group acting isometrically and properly discontinuously on $S(\H)$. Then $S(\H)/\Gamma$ is homogeneous if and only if
$\H$ is a Hilbert space over
$\mathbb F$ where $\mathbb F$ is one of the fields $\R$, $\C$ and $\mathbb Q$ (quaternions), $\Gamma$ is a finite multiplicative group of elements of norm $1$ in $\mathbb F$ which is not contained in a proper subfield $\mathbb F_1$, $\R \subset \mathbb F_1 \subseteq \mathbb F$
of $\mathbb F$ and $\Gamma$ acts on $S(\H)$ by $\mathbb F$-scalar multiplication of vectors. Conversely, all such manifolds are homogeneous manifolds of constant sectional curvature $1$.
\end{thm}
\begin{proof}
Let $T$ be a Clifford translation. By Lemma \ref{geodesic}, $T$ leaves  a geodesic $\gamma$ invariant on $S(\H)$. Now, it is well-known that any geodesic in $S(\H)$ is closed and it spans a two dimensional vector space. Hence $T$ must act as a Clifford Translation on the unit circle and so it has finite order $n$ \cite{wolf}. Moreover, $T$ is a Clifford translation on the unit sphere of any subspace spanned by $x,T(x),\ldots, T^{n-1} (x)$ and so it acts by multiplication on it. Now,
$\H=W\oplus W^\perp$ and we can iterate the above argument on $W^\perp$. However,  the \emph{displacement function} of $T$ must be constant and so $T$ must act by multiplication on  $\H$ and so the Theorem follows directly by a Theorem of Wolf \cite{wolf1,wolf}.
\end{proof}
We conclude this section  giving a properly discontinuous $\Z^n$ action on the hyperbolic space of infinite dimensional.

Let $\mathbb H^n=\{(x_1,\ldots,x_{n+1}):\, x_{n+1}>0 \}$ and let $g=\frac{1}{x_{n+1}^2} (\mathrm d x_1^2 + \cdots + \mathrm d x_{n+1}^2)$.  $\riem{\mathbb H^n}{g}$
is the upper half plane model of the hyperbolic space. The group $\Z^n$ acts isometrically and properly discontinuously on $\H^n$ by setting
\[
\bigl( (m_1,\ldots,m_n),(x_1,\ldots,x_{n+1})\bigr) \mapsto (x_1+m_1,\ldots,x_n+m_n,x_{n+1}).
\]
It is well-known that $\mathbb H^n$ can be also described as
\[
\{(x_1,\ldots,x_n,t)\in \R^{n+1}:\,t>0, \  x_1^2+\cdots +x_n^2-t^2 =-1 \}.
\]
The Riemannian metric on $\mathbb H^n$, that we also denote with $g$, is the restriction of the Minkowski  metric of $\R^{n+1}$ on $\mathbb H^n$. The isometry group $\mathrm{Iso}(\H^n,g)=\mathrm{O}(n,1)_+$ is  the set of the isometries with respect to the Minkowski metric which preserve $\H^n$. Therefore, we have an injective homomorphism $\Z^n \hookrightarrow \mathrm{O}(n,1)$ such that $\Z^n$ acts linearly, isometrically and properly discontinuously on the hyperbolic space $\H^n$.

Let $\H=l_2 (\Z^n)$. Then
\[
\{\bigl( x, (x_1,\ldots,x_n , t)\bigr) \in l_2 (\Z^n) \times \R^{n+1}:\, t>0, \  \parallel x \parallel^2 + x_1^2+\cdots +x_n^2-t^2 =-1 \}
\]
endowed with the Riemannian metric  induced by the Minkowski metric on $l_2 (\Z^n) \times \R^{n+1}$, given by:
\[
\langle \bigl(x,(x_1,\ldots,x_n , t)\bigr), \bigl(y,(y_1,\ldots,y_n , s)\bigr)=\langle x,y \rangle + x_1y_1+\cdots +x_ny_n-st,
\]
is the infinite dimensional hyperbolic space that we will denote by $\H^\infty$.
Then $\Z^n$ acts on $\H^\infty$ by setting:
\[
\Z^n \times \H^\infty \lra \H^\infty , \ \ \bigl( n,(x_1,\ldots,x_n,t) \bigr)\mapsto \bigl(n(x), n(x_1,\ldots,x_n,t)\bigr),
\]
where the $\Z^n$ action on $l_2 (\Z^n)$  has been described in section \ref{proof1} while the $\Z^n$ action on the last coordinate is the $\Z^n$ action on $\R^{n+1}$ described before.
Indeed, $\Z^n$ preserves the Minkowski metric and so it induces an isometric action on $\H^\infty$. We shall prove that this action is properly discontinuous on $\H^\infty$. It is easy to check that $\Z^n$ acts freely on $\H^\infty$.
Let $a_k$ be a sequence in $\Z^n$ and let
$(x,x_1,\ldots,x_n,t)\in \H^\infty$.  Assume that the sequence $a_k \bigl(x,(x_1,\ldots,x_n,t)\bigr)$ converges to an element of $\H^\infty$. Then $a_k (x)$ must converge in $l_2 (\Z^n)$ and so by Theorem \ref{miotesi}, this implies that $x=0$.
Hence $(x_1,\ldots,x_n, t)\in \H^n$ and $a_k (x_1,\ldots,x_n ,t)$ converges which is a contradiction. Applying Proposition \ref{dp} and Theorem \ref{hp} we get the following result.
\begin{prop}
There exists a complete Hilbert manifold modeled on a infinite dimensional separable Hilbert space with constant sectional curvature $-1$ whose fundamental group is isomorphic to $\Z^n$. In particular this manifold is not homogeneous.
\end{prop}
\begin{remark}
Let $n\in \N$. For any positive integer $p$, let $\H=l_2 (\Z^n)\times  \underbrace{\R^{n+1} \times \cdots \times \R^{n+1}}_p$
endowed by the symmetric bilinear form of index $p$ whose  quadratic form is given by:
\[
\begin{split}
q \bigl((x,( x^1_1,\ldots, x_n^1, t^1),\ldots, (x_1^p,\ldots ,x_n^p,t^p) \bigr)&=\parallel x \parallel^2 \\ &+ \sum_{k=1}^p ( (x_1^k)^2 +\ldots +( x_n^k)^2 -(t^k )^2 ).
\end{split}
\]
Then $\H^\infty_p=\{z\in \H:\ q(z)=-1, \ t^1,\cdots ,t^p >0\, \}$ is a Hilbert manifold that can be endowed  by a Semi-Riemannian metric of index $p-1$.
Now, $\Z^n$ acts diagonally on $\H$ and it preserves $\H^\infty_p$. By the same arguments used before, it can be proved that $\Z^n$ acts isometrically and properly discontinuously on $\H^\infty_p$.
Therefore $\H^\infty_p/\Z^n$ is a Hilbert manifold endowed by a Semi-Riemannian metric of index $p-1$ whose fundamental groups is isomorphic to $\Z^n$.
\end{remark}

\end{document}